\documentclass[11pt]{amsart}
\usepackage{amssymb}
\usepackage{amsmath}
\usepackage{amsfonts}
\usepackage{graphicx}
\theoremstyle{plain}

\newtheorem{thm}{Theorem}[section]
\newtheorem{prop}[thm]{Proposition}

\newtheorem{cor}[thm]{Corollary}

\newtheorem{rmk}[thm]{Remark}

\newtheorem{prb}[thm]{Problem}
\theoremstyle{remark}

\def\pmc#1{\setbox0=\hbox{#1}
    \kern-.1em\copy0\kern-\wd0
    \kern.1em\copy0\kern-\wd0}

\begin{document}

\title
{On the homology of the Harmonic Archipelago}


\author[U.~H.~Karimov]{Umed H.~Karimov}
\address{Institute of Mathematics,
Academy of Sciences of Tajikistan, Ul. Ainy $299^A$, Dushanbe
734063, Tajikistan} \email{umedkarimov@gmail.com}

\author[D.~Repov\v s]{Du\v san Repov\v s}
\address{Faculty of Education and
Faculty of Mathematics and Physics,
University of Ljubljana, P.O.Box 2964,
Ljubljana 1001, Slovenia} \email{dusan.repovs@guest.arnes.si}

\begin{abstract}
We calculate the singular homology and \v Cech cohomology groups
of the {Harmonic archipelago}. As a
corollary, we prove that this space is not homotopy equivalent to
the Griffiths space. This is interesting in view of Eda's proof that the first singular
homology groups of these spaces are isomorphic.
\end{abstract}

\date{\today}
\maketitle
\date{\today}
\section{Introduction}

The following interesting problem from contemporary theory of
Peano continua and combinatorial group theory has been widely
discussed and investigated (most recently at the 2011 workshop on
wild topology in Strobl, Austria \cite{C}), because it concerns
two well-known and important $2$-dimensional spaces,
namely the
{\it Griffiths} space $\mathcal{G}$ and the {\sl Harmonic}
archipelago ${\mathcal{HA}}$ (cf. \cite{C}):

\begin{prb}\label{Main problem}
Are the fundamental groups of
the Griffiths space $\mathcal{G}$ and the Harmonic archipelago
${\mathcal{HA}}$
isomorphic?
\end{prb}

This difficult problem remains open. Its solution will require
a
deep
understanding
of the structure of the fundamental groups of
these
spaces.
In the present
paper, which is a step in this direction,
we shall investigate
the {\it abelianization} of
the fundamental group of the Harmonic archipelago
${\mathcal{HA}}$.

It is well
known (cf. e.g. \cite[Theorem 2A.1]{H})
that the $1$-dimensional singular homology group with integer
coefficients $H_1 (X;\mathbb Z)$ of a path-connected space $X$
is isomorphic to
the abelianization of the fundamental group $\pi_1(X)$ of
$X$:
$$H_1(X;\mathbb Z) \cong \pi_1(X)/[\pi_1(X), \pi_1(X)].$$
Our first result is
based on the structure of the homology groups of the
Hawaiian earring
$\mathbb{H}$
(an alternative proof, using infinitary words, was
given by Eda \cite{E1}):

\begin{thm}\label{Homology.}
Let $\mathcal{HA}$ denote the Harmonic archipelago.
Then
$$H_1(\mathcal{HA};\mathbb Z)
\cong
(\prod_{i\in \mathbb{N}} \mathbb{Z})/(\sum_{i\in
\mathbb{N}}\mathbb{Z}),$$
whereas
$H_n(\mathcal{HA};\mathbb Z)\cong 0,$
for all $n\ge
 2.$
\end{thm}

Eda proved \cite{E1} that the Griffiths space $\mathcal{G}$ and
the Harmonic archipelago $\mathcal{HA}$ have isomorphic
$1$-dimensional singular homology groups,
$$H_{1}(\mathcal{G}; \mathbb Z)\cong
H_{1}(\mathcal{HA}; \mathbb Z).$$

Now, it is well known that the Griffiths space $\mathcal{G}$ is
cell-like and therefore it has trivial \v Cech cohomology groups,
$\check{H}^{*}(\mathcal{G};\mathbb Z)\cong
\check{H}^{*}(pt;\mathbb Z).$

On the other hand, by our second main result stated below, the \v
Cech cohomology of the Harmonic archipelago $\mathcal{HA}$ does
not vanish:


\begin{thm}\label{H^2(HA)}
Let $\mathcal{HA}$ denote the Harmonic archipelago. Then
$$\check{H}^2(\mathcal{HA}; \mathbb Z)\cong (\prod_{i\in \mathbb{N}}
\mathbb{Z})/(\sum_{i\in \mathbb{N}}\mathbb{Z}),$$ whereas
$\check{H}^n(\mathcal{HA}; \mathbb Z)\cong 0,$ for all $n \neq 0,
2.$
\end{thm}

As an immediate consequence we obtain the following important corollary:

\begin{cor}
The Griffiths  space
$\mathcal{G}$
and
the Harmonic archipelago
$\mathcal{HA}$
are not homotopy equivalent.
\end{cor}

\section{Preliminaries}

The constructions of
the Griffiths space $\mathcal{G}$ and the
Harmonic archipelago $\mathcal{HA}$
are based on the {\it Hawaiian
earring} $\mathbb{H}$, a
classical $1$-dimensional planar
Peano continuum: 
$$\mathbb{H} = \bigcup_{n\in \mathbb{N}} \{(x,y) \in \mathbb{R}^2 |\ x^2 + (y - \frac{1}{n})^2 =
(\frac{1}{n})^2\}.$$

The {\it Griffiths space} $\mathcal{G}$ is a one-point union of
two cones over the Hawaiian earring $\mathbb{H}$
(cf. \cite{G}): 
Consider two copies
of the Hawaiian earring
$\mathbb{H}$
in $\mathbb{R}^2 \times \{0\} \subset \mathbb{R}^3$
$$\mathbb{H}_+ = \bigcup_{n\in \mathbb{N}} \{(x,y, 0) \in \mathbb{R}^3 |\ x^2 + (y - \frac{1}{n})^2 =
(\frac{1}{n})^2\}$$ and $$\mathbb{H}_-= \bigcup_{n\in \mathbb{N}}
\{(x,y, 0) \in \mathbb{R}^3 |\ x^2 + (y + \frac{1}{n})^2 =
(\frac{1}{n})^2\}.$$

Let $C(\mathbb{H}_+, (0, 0, 1))$ and $C(\mathbb{H}_-, (0, 0, -1))$
be two cones on the spaces $\mathbb{H}_+$ and $\mathbb{H}_-$  with
vertices at the points $(0, 0, 1)$ and $(0, 0, -1),$ respectively.
The Griffiths space $\mathcal{G}$ is then defined as the following
subspace of $\mathbb{R}^3:$
$$\mathcal{G} = C(\mathbb{H}_+, (0, 0, 1)) \cup C(\mathbb{H}_-, (0, 0, -1)).$$
Griffiths \cite{G}
proved that the fundamental group of this one-point
union of contractible spaces is nontrivial.

The {\sl Harmonic archipelago} $\mathcal{HA}$ was introduced by
Bogley and Sieradski \cite{BS}.
It can be simply
described as follows:
$\mathcal{HA}$ is a noncompact
space which
is obtained by adjoining a sequence of 'tall' disks
between consecutive loops of the Hawaiian earring (cf. \cite{M}).

The following proposition
will be 
useful in the
sequel (cf. e.g.
\cite[Corollary 0.21]{H} or
\cite[Theorem 1.4.13]{S}).

\begin{prop}\label{deform}
{\it Spaces $X$ and $Y$ are
homotopy equivalent if and only if there is a space $Z$,
containing both $X$ and $Y$ as deformation retracts}.
\end{prop}

Let
$$C_n = \{(x,y, 0)\in \mathbb{R}^3 |\ x^2 + (y-\frac{1}{n})^2 =
(\frac{1}{3n(n+1)})^2\}_{n\in \mathbb{N}}$$
be a countable number
of circles and $\theta = (0, 0, 0)$ the origin of
$\mathbb{R}^3$.
It follows by Proposition~\ref{deform} that the
Harmonic archipelago $\mathcal{HA}$ is homotopy equivalent to the
following subspace of $\mathbb{R}^3$ consisting of all cones
$C(C_n, (0, \frac{1}{n}, 1))$ over
the circles $C_n$, with the
vertices at the points $(0, \frac{1}{n}, 1),$ $n \in \mathbb{N},$
connected by the segments and the point $\{\theta\}$, which we
shall denote by $\emph{HA}$ and call the {\it Formal Harmonic
archipelago}:
$$\emph{HA} = \bigcup_{n = 1}^{\infty} C(C_n, (0,
\frac{1}{n}, 1))\cup$$
$$\bigcup_{n = 1}^{\infty}\{(0, y, 0)| y \in [
\frac{3n+7}{3(n+1)(n+2)}, \frac{3n+2}{3(n)(n+1)}]\} \cup
\{\theta\}.$$

The {\it Modified Hawaiian earring} $\mathcal{MH}$
is defined as follows:
$$\mathcal{MH} = \bigcup_{n = 1}^{\infty} C_n \cup \bigcup_{n =
1}^{\infty}\{(0, y, 0)| y \in [\frac{3n+7}{3(n+1)(n+2)},
\frac{3n+2}{3(n)(n+1)}]\} \cup \{\theta\}.$$

\begin{figure}[fig1]
\includegraphics[width=1.00\textwidth]{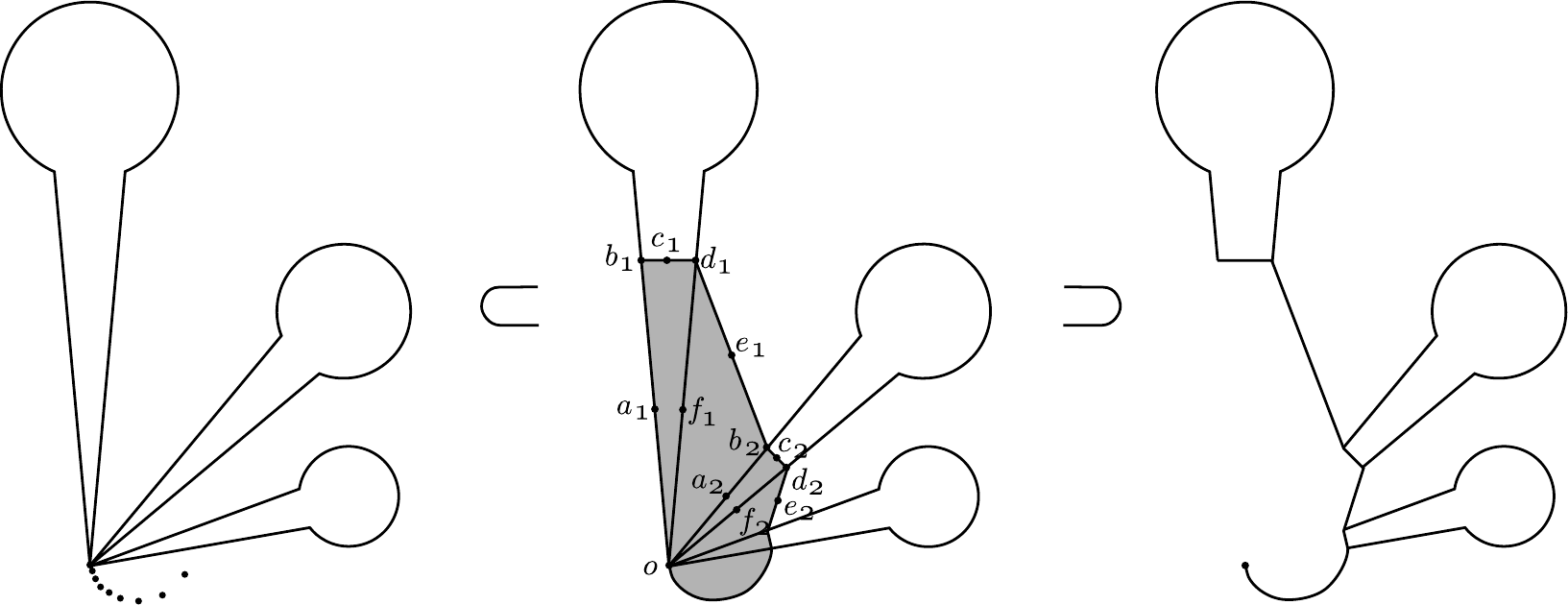}
\label{Supportfigure}
\caption{Homotopy representatives of the Hawaiian earring $\mathbb{H}$.}
\end{figure}

The Modified Hawaiian earring
$\mathcal{MH}$
is
homotopy equivalent to the Hawaiian
earring,
$\mathcal{MH} \simeq \mathbb{H}.$ Indeed, both of these
spaces are deformation retracts of the third one, as indicated
in the middle of
Figure 1  (all points $a_n, b_n, c_n, d_n, e_n,$ and
$f_n$ converge
to the point $o$).
The piecewise linear deformation which
moves the points $c_n$ and $e_n$
to the point $o$,
and fixes the
points $b_n$ and $d_n$,
yields a
space homeomorphic to the Hawaiian
earring $\mathbb{H}$.

The piecewise
linear deformation which moves
sequentially the points $a_n$ to the points $d_n,$
the segments
$[f_n, d_n]$ to $[e_n, d_n]$,
and the segments $[o, f_n]$ to the
line $[o, b_{n+1}] \cup [b_{n+1}, e_n]$,
with  fixed points $o, b_n, c_n,
d_n$, yields  the space $\mathcal{MH}$, therefore
by Proposition~\ref{deform}, $U$
is homotopy equivalent to the
Hawaiian
earrring $\mathbb{H}$.

\begin{figure}[fig2]
\includegraphics[width=1.00\textwidth]{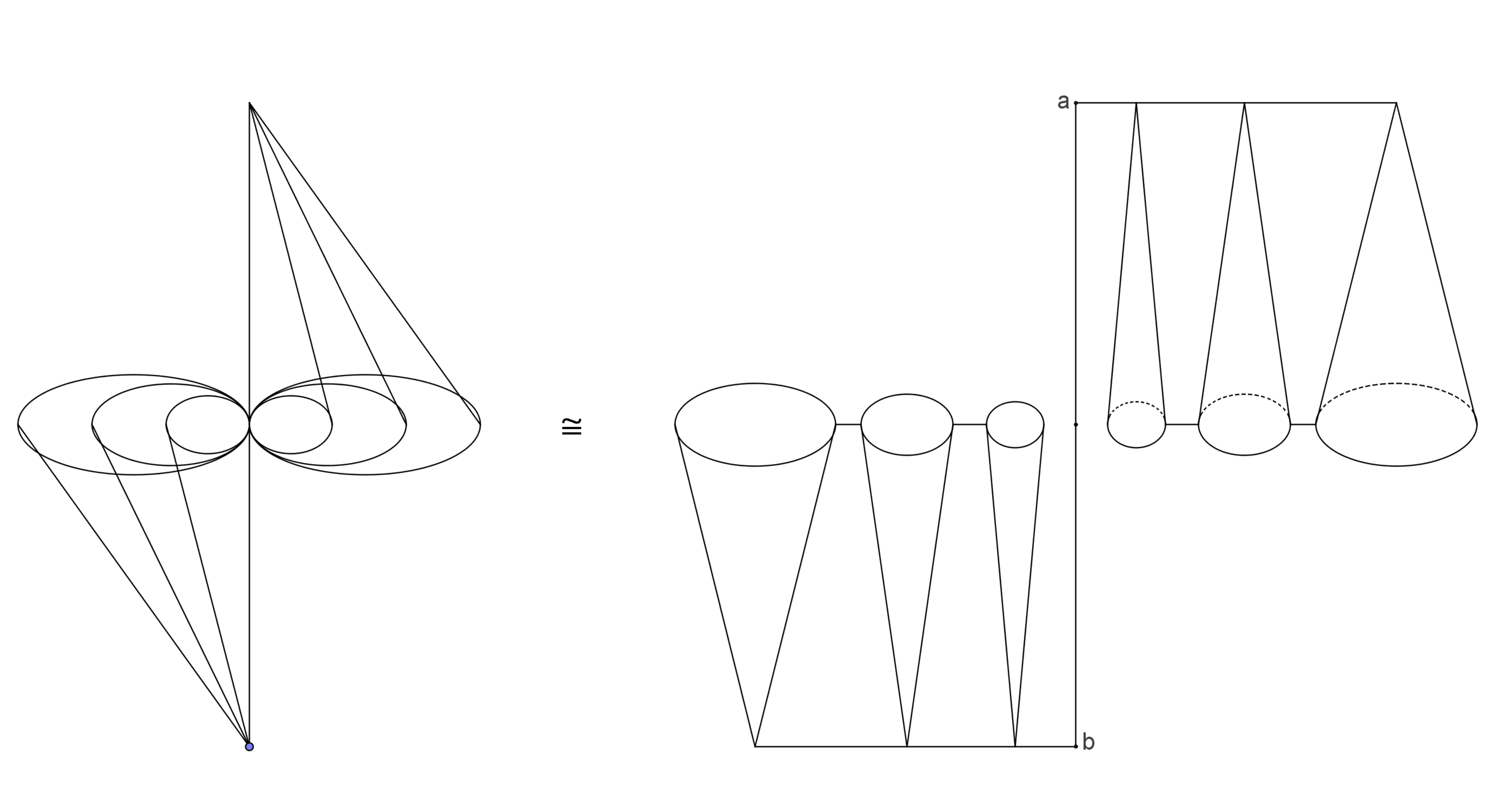}
\label{Supportfigure}
\caption{The Griffiths space $\mathcal{G}$ and the Modified Griffiths space $\mathcal{MG}.$}
\end{figure}

Let us define the {\sl Modified Griffiths space} $\mathcal{MG}.$
To this end let us introduce some new spaces.
Let $\emph{HA}_-$ be
the space symmetric in
$\mathbb{R}^3$
to $\emph{HA}$,
with respect to the point
$\theta$:
$$\emph{HA}_- = \bigcup_{n = 1}^{\infty} C(C_n^{-}, (0,
\frac{-1}{n}, -1))$$
$$\cup \bigcup_{n = 1}^{\infty}\{(0, y, 0)| -y
\in [ \frac{3n+7}{3(n+1)(n+2)}, \frac{3n+2}{3(n)(n+1)}]\} \cup
\{\theta\},$$ where
$$C_n^- = \{(x,y, 0)\in \mathbb{R}^3 |\ x^2 + (y+\frac{1}{n})^2 =
(\frac{1}{3n(n+1)})^2\}_{n\in \mathbb{N}}.$$

Define the {\sl convex hull} $L(M)$ of a subset $M$ of
$\mathbb{R}^3$ as the intersection of all convex sets in
$\mathbb{R}^3$ containing the set $M$.
For $n\in \mathbb{N}$, the sets $F^+_n$ and
$F^-_n$  are defined as the convex hulls of
the quadruples of points of $\mathcal{R}^3$ as follows:
$$F^+_n = L(\{(0, \frac{3n+7}{3(n+1)(n+2)}, 0),\ (0,
\frac{3n+2}{3n(n+1)}, 0),\ (0, \frac{1}{n+1}, 1),\ (0,
\frac{1}{n}, 1)\})$$
and
$$F^-_n = L(\{(0,
\frac{-3n-7}{3(n+1)(n+2)}, 0),\ (0, \frac{-3n-2}{3n(n+1)}, 0),\
(0, \frac{-1}{n+1}, -1),\ (0, \frac{-1}{n}, -1))\}),$$
respectively. The Modified Griffiths space
$\mathcal{MG}$
is then
defined as the
following subspace of $\mathbb{R}^3:$
$$\mathcal{MG} = \emph{HA}\cup \emph{HA}_-\cup F^+_n\cup F^-_n\cup L(\{(0, 0, -1), (0, 0,
1)\}),$$
where $L(\{(0, 0, -1), (0, 0, 1)\})$ is a
compact segment
in $\mathbb{R}^3$ with end points at $(0, 0, 1)$ and $(0, 0,
-1),$ see Figure 2.

Now, let us also define the {\it Modified Harmonic archipelago
$\mathcal{MHA}$}. Let $a = (0, 0, 1)$ and $b = (0, 0, -1)$ be two
points of the $\mathcal{MG}$ and set
$$ \mathcal{MHA} = \mathcal{MG}  \setminus \{a, b\}.$$

\begin{figure}[fig3]
\includegraphics[width=1.20\textwidth]{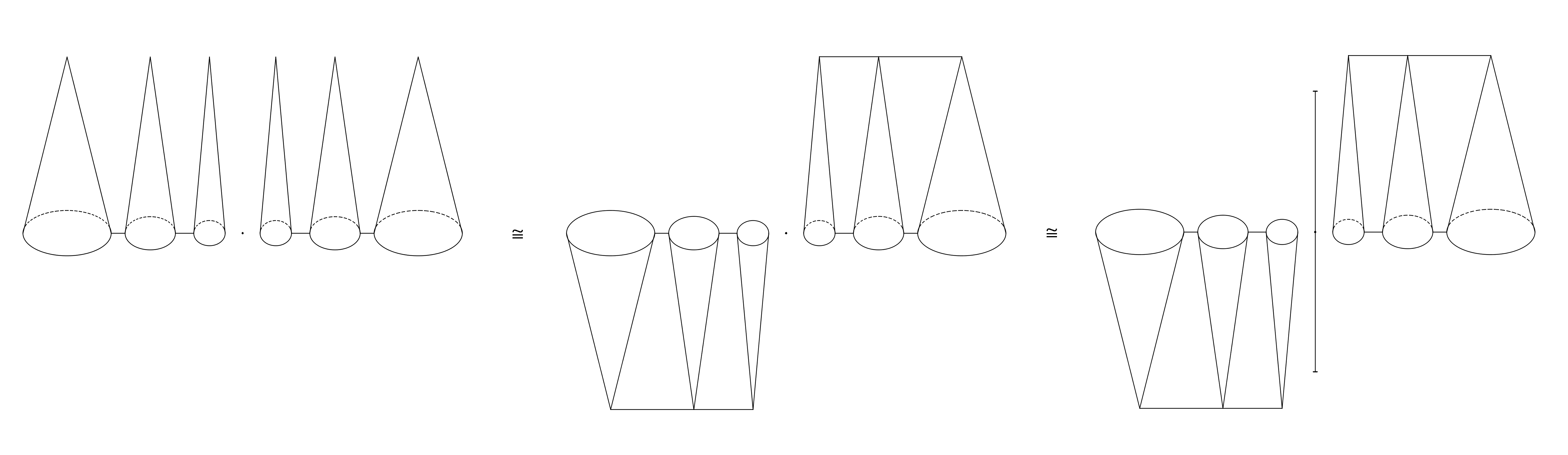}
\label{Supportfigure}
\caption{Homotopy representatives of the Harmonic archipelago $\mathcal{HA}.$}
\end{figure}

\begin{prop}{\label{Splitting}}
{\it Suppose that in the short exact sequence
$$0\longrightarrow
A\stackrel{\alpha}\longrightarrow B
\stackrel{\beta}\longrightarrow C \longrightarrow 0$$
there exists
a
projection $p:B\to A$ (i.e. $p$ is a homomorphism such
that $p\circ \alpha:A\to A$ is the identity mapping $1_A$).
Then there exists an isomorphism $\varphi :B \rightarrow A \oplus
C$ 
such that
$\varphi (b) = (p(b), \beta (b)),$ for every $b\in B$.}
\end{prop}

In this case, the short exact sequence $0\longrightarrow
A\stackrel{\alpha}\longrightarrow B
\stackrel{\beta}\longrightarrow C \longrightarrow 0$
is said {\sl to split} \cite[Splitting Lemma, p.147]{H} or is {\sl
splitting} (cf. e.g. \cite[Lemma 9.1 and p. 38 ]{F1}).

The following statement is well known, see e.g. \cite{F1}.

\begin{prop}{\label{Ext}}
If the group $A$ is algebraically compact and $C$ is torsion-free
then every exact sequence $0\longrightarrow
A\stackrel{\alpha}\longrightarrow B
\stackrel{\beta}\longrightarrow C \longrightarrow 0$ splits.
\end{prop}

Throughout this paper only
singular homology  $H_*$ and \v{C}ech
 cohomology $\check{H}^*$  with integer coefficients
will be used.

The following statement is a reformulation of a theorem of Eda and
Kawamura \cite{EK} (they used the notion of $p-$adic completion).

\begin{prop}{\label{H_1(H)}}
For the $1$-dimensional singular homology group of the Hawaiian
earrings $H_1(\mathbb{H})$ there exists the following exact
sequences which splits:
$$0 \rightarrow
(\prod_{i\in \mathbb{N}}\mathbb{Z}) / 
(\sum_{i\in \mathbb{N}}\mathbb{Z})
\rightarrow
H_1(\mathbb{H})
\stackrel{\sigma}\rightarrow
\prod_{i\in \mathbb{N}}\mathbb{Z}
\rightarrow 0$$
or, equivalently
$$0 \rightarrow
(\sum_{i\in c} \mathbb{Q})\oplus \prod_{p \ \hbox{\rm prime}
}(\prod_{i\in
\mathbb{N}}\mathbb{J}_p )\rightarrow H_1(\mathbb{H})
\stackrel{\sigma}\rightarrow \prod_{i\in
\mathbb{N}}\mathbb{Z}\rightarrow 0.$$
\end{prop}

\begin{proof}
It was proved in \cite{EK}  that there exists the following exact
sequence which is splitting:
$$0 \rightarrow
(\prod_{p \ \hbox{\rm prime}}A_p)
\oplus (\sum_{i\in c} \mathbb{Q}) \rightarrow H_1(\mathbb{H})
\stackrel{\sigma}\rightarrow \prod_{i\in
\mathbb{N}}\mathbb{Z}\rightarrow 0$$  where $A_p$ is the $p-$adic
completion of the direct sum of $p-$adic integers
$\oplus_c\mathbb{J}_p,$ and $c$ is the continuum cardinal.
According to
a theorem of Balcerzyk \cite{B} \cite[VII.42,\ Exercise 7]{F1},
we have 
$$(\prod_{i\in \mathbb{N}}\mathbb{Z}) /
(\sum_{i\in \mathbb{N}}\mathbb{Z})
\cong 
(\prod_{p \ \hbox{\rm prime}}A_p)
\oplus (\sum_{i\in
c} \mathbb{Q}).$$ Therefore the first desired isomorphism follows.

The second isomorphism again
follows from Theorem 3.1 of \cite{EK}
and Theorem 40.2 of \cite{F1}.
\end{proof}

\begin{prop}{\label{Inverse Lim}}\rm{(}cf. \cite{H}\rm{)}
\it{Let the space $X$ be a countable union of an increasing system
of open sets $\{U_i\}_{i\in \mathbb N}.$
Then for the \v{C}ech cohomology $\check{H}^*$
there exists the following exact sequence:
$$0\rightarrow \underleftarrow{\rm{lim}}^{(1)}\check{H}^{n-1}(U_i) \rightarrow \check{H}^n(X)
\rightarrow \underleftarrow{\rm{lim}}\ \check{H}^n(U_i)\rightarrow
0,$$ where $\underleftarrow{\rm{lim}}^{(1)}$ is the first derived
functor of the inverse limit functor
$\underleftarrow{\rm{lim}}.$}
\end{prop}

\section{Proofs of the Main Theorems}

\subsection{Proof of Theorem~\ref{Homology.}}
\begin{proof}
Since the Harmonic archipelago
$\mathcal{HA}$
is homotopy equivalent to the
Formal Harmonic archipelago
$\emph{HA},$
it suffices to calculate the group
$H_1(\emph{HA})$. Let $U$ and $V$ be the open sets defined as
follows:
$$U = \emph{HA} \cap \{(x, y, z)\in \mathbb{R}^3|\
z\in [0, \frac{2}{3})\}, \ V = \emph{HA} \cap \{(x, y, z)\in
\mathbb{R}^3 |\ z > \frac{1}{3}\}.$$
Consider the following part
of the Mayer-Vietoris sequences of the triad $(HA, U, V)$:
$$H_2(HA)\rightarrow H_1(U \cap V)\stackrel{i}\rightarrow H_1(U)\oplus H_1(V)
\stackrel{j}\rightarrow$$
$$H_1(\emph{HA})\stackrel{\delta}\rightarrow
H_0(U\cap V)\rightarrow H_0(U)\oplus H_0(V).$$
Obviously, $V$ is
homotopy equivalent to a
countable discrete union of points and the
space $U\cap V$ is homotopy equivalent to a
discrete countable
union of circles.

The homomorphism $H_0(U\cap V)\rightarrow
H_0(V)$ is an
isomorphism since the 0-dimensional homology group is
isomorphic to the direct sum of $Z$ cardinality of
path-connectedness components  \cite[Theorem 4.4.5]{S}.
The space
$U$ is homotopy equivalent to  $\mathcal{MH}$ since
$\mathcal{MH}$ is a
deformation retract of $U$  and $\mathcal{MH}
\simeq \mathbb{H}$,
therefore $U \simeq \mathbb{H}.$

The singular homology groups are homology groups with compact
support, therefore $H_n(HA) \cong \underrightarrow{\rm{lim}}\
H_n(P),$ where $P$ are Peano subcontinua of $HA$
(cf. \cite[Theorem
4.4.6]{S}).
Obviously, there exists a confinal sequence of Peano
continua such that every $P$ is homotopy equivalent to the
Hawaiian earring
$\mathbb{H}$
and $H_n(P) \cong 0$ for all $n > 1$ (by \cite{CF},
$H_n$ are trivial for one-dimensional spaces, for all $n > 1$),
therefore $H_n(HA) \cong 0,$
for all $n > 1.$ In particular,
$H_2(HA)
\cong 0.$

Therefore we have following commutative diagram with exact rows
and columns:
$$\begin{array}{ccccccccc}
&& && 0 && 0 &&    \\
&& && \downarrow && \downarrow &&    \\
&& && Ker (\sigma) &\stackrel{\cong}\rightarrow & Ker( \sigma) &&    \\
&& && \varphi_1\downarrow && \varphi_2\downarrow &&    \\
0 &\rightarrow &\sum_{i\in \mathbb{N}}\mathbb{Z}
&\stackrel{\varphi_3}\rightarrow & H_{1}(\emph{H})
&\stackrel{\varphi_4}\rightarrow & H_{1}(\emph{HA})
&\rightarrow & 0\\
&& \downarrow \cong && \sigma\downarrow && \varphi_5\downarrow &&    \\
0 &\rightarrow &\sum_{i\in \mathbb{N}}\mathbb{Z}
&\stackrel{\varphi_6}\rightarrow & \prod_{i\in
\mathbb{N}}\mathbb{Z} &\stackrel{\varphi_7}\rightarrow &
(\prod_{i\in \mathbb{N}} \mathbb{Z})
/(\sum_{i\in \mathbb{N}}\mathbb{Z}) &\rightarrow & 0\\
&& && \downarrow && \downarrow &&    \\
&& && 0 && 0 &&
\end{array}$$
in which homomorphisms $\varphi_3$ and $\varphi_4$  correspond to
$i$ and $j$, respectively. The homomorphism $\sigma$ 
is defined for any
element $[l]$ of the $H_1(\mathbb{H})$
as $(l_1, l_2,
l_3, \dots),$ where $l_i$ is the winding number of the loop $l$
around the $i-$th circle $S_i$ (cf. \cite[p. 310]{EK}).
It follows that
the composition $\sigma \varphi_3,$ which we can identify with
$\varphi_6,$ is a monomorphism and $\rm{Im}(\varphi_3)\cap
\rm{Ker}(\sigma) = 0.$ Then  the composition $\varphi_1\varphi_4$
is a monomorphism which we shall identify with $\varphi_2$.

Homomorphism $\varphi_7$ is the quotient mapping. The homomorphism
$\varphi_5$ is defined as follows. Take any element $a\in
H_1(HA.)$ Due to the exactness of the middle row there exists an
element $b\in H_1(H)$ such that $\varphi_4(b) = a.$ Define
$\varphi_5(a) \equiv \varphi_7\sigma(b).$

Let us show that this mapping is
well-defined. Suppose that $\varphi_4(b') = a$. Then the difference
$b - b'$ belongs to  $\rm{Im}(\varphi_3)$ and $\sigma (b -
b')\in \rm{Im}(\varphi_6)$, therefore $\varphi_7(b - b') = 0.$
This means that $\varphi_7(b) = \varphi_7(b')$ and the mapping
$\varphi_5$ is well-defined.

If $a = a_1 + a_2$ then there exist $b_1$ and $b_2$ such that
$\varphi_4(b_1) = a_1$ and $\varphi_4(b_2) = a_2.$ Since $\sigma$
and $\varphi_7$ are homomorphisms we have $\varphi_5(a) =
\varphi_7\sigma(b_1 + b_2) = \varphi_7\sigma(b_1 ) +
\varphi_7\sigma(b_2) = \varphi_5(a_1) + \varphi_5(a_2)$ and
$\varphi_5$ is
indeed a homomorphism. Since $\sigma$ and
$\varphi_7$ are epimorphisms it follows that $\varphi_5$ is an
epimorphism. The composition $\varphi_5\varphi_2$ is trivial since
the composition $\sigma\varphi_1$ is the zero homomorphism. If
$a\in H_1(HA) $ such that $\varphi_5(a) = 0$ then
$\varphi_7\sigma(b) = 0,$ for any $b$
for the corresponding $a.$

Choose
one of these elements $b.$ It follows that there exists $c $ such
that $\varphi_6(c) = \sigma (b)$. Since the left projection in the
diagram is
an isomorphism there exists $c'$ such that $\sigma
(\varphi_3 (c')) = \sigma (b) .$ It follows that $b - \varphi_3
(c') = \varphi_1(d)$, for some element $d.$ Then $\varphi_4(b -
\varphi_3 (c')) = \varphi_4\varphi_1(d)$, but $\varphi_4(\varphi_3
(c')) = 0 $ and $\varphi_4(b) = a$, therefore $a = \varphi_2(d)$
 and the right column is an
exact sequence.

By  Proposition \ref{H_1(H)} we have that $Ker (\sigma) \cong
\prod_{i\in \mathbb{N}}\mathbb{Z}/\sum_{i\in
\mathbb{N}}\mathbb{Z}.$ By  Proposition \ref{Ext}, the right
column splits since the group $\prod_{i\in
\mathbb{N}}\mathbb{Z}/\sum_{i\in \mathbb{N}}\mathbb{Z}$ is
algebraically compact  and
torsion-free \cite[Corollary 42.2]{F1} .

Therefore we have the following
exact sequence which splits:
$$0 \rightarrow 
(\prod_{i\in \mathbb{N}} \mathbb{Z})
/ (\sum_{i\in \mathbb{N}}\mathbb{Z})
\rightarrow
H_{1}(\emph{HA})\stackrel{p}\rightarrow
(\prod_{i\in \mathbb{N}} \mathbb{Z})
/ (\sum_{i\in \mathbb{N}}\mathbb{Z})
\rightarrow 0$$
and hence:
$$H_{1}(\emph{HA})\cong
(\prod_{i\in \mathbb{N}} \mathbb{Z})
/ (\sum_{i\in \mathbb{N}}\mathbb{Z})
\oplus
(\prod_{i\in \mathbb{N}} \mathbb{Z})
/ (\sum_{i\in \mathbb{N}}\mathbb{Z}).$$
However,
obviously
$$(\prod_{i\in \mathbb{N}} \mathbb{Z})
/ (\sum_{i\in \mathbb{N}}\mathbb{Z})
\oplus
(\prod_{i\in \mathbb{N}} \mathbb{Z})
/ (\sum_{i\in \mathbb{N}}\mathbb{Z})
\cong
(\prod_{i\in \mathbb{N}} \mathbb{Z})
/ (\sum_{i\in \mathbb{N}}\mathbb{Z}),$$
therefore 
$$H_1(\mathcal{HA})\cong H_1(HA)\cong
(\prod_{i\in \mathbb{N}} \mathbb{Z})
/ (\sum_{i\in \mathbb{N}}\mathbb{Z}).$$

\end{proof}

\subsection{Proof of  Theorem~\ref{H^2(HA)}}

\begin{proof}

Let $U_1 = U,$ (where $U$ was defined in the proof of
Theorem~\ref{Homology.})  and $U_i$ at $i > 1$ be the following
open subspaces of $\emph{HA}$:
$$U_i =
\emph{HA} \cap \{(x, y, z)\in \mathbb{R}^3|\ y > \frac{2i +
1}{2i(i + 1)}\}\cup U.$$

Obviously, $U_i$ is homotopy equivalent to the Modified Hawaiian
earring and therefore to the Hawaiian earring. Since the Hawaiian
earring can be presented as the
inverse limit of bouquets of finite
numbers of cycles:
$$S^1_{1}\stackrel{\pi_1}\longleftarrow S^1_{1} \bigvee S^1_{2}\stackrel{\pi_2}\longleftarrow
\cdots \stackrel{\pi_{n-1}}\longleftarrow \bigvee_{j=1}^{n}S^1_{j}
\stackrel{\pi_n}\longleftarrow \bigvee_{j=1}^{n+1}S^1_{j}\cdots,$$
where the
projections $\pi_n$ map the corresponding circles
$S_{n+1}^1$ to the base point of the bouquets and map all other
circles identically, it follows that the $1$-dimensional \v{C}ech
cohomology of the Hawaiian earring is isomorphic to $\sum_{j =
1}^{\infty}\mathbb{Z}$.
Therefore $\check{H}^1 (U_i) \cong \sum_{j
= 1}^{\infty}\mathbb{Z}.$ By Proposition~\ref{Inverse Lim}, we
have the following exact sequences:
$$0\rightarrow \underleftarrow{\rm{lim}}^{(1)}\check{H}^{1}(U_i) \rightarrow
\check{H}^2(HA)\rightarrow \underleftarrow{\rm{lim}} \
\check{H}^2(U_i)\rightarrow 0.$$

The embedding $U_i \subset U_{i + 1}$ generates the monomorphism
$\check{H}^{1}(U_i)\leftarrow \check{H}^{1}(U_{i+1})$ which we can
identify with $\sum_{j = 1}^{\infty}\mathbb{Z}
\stackrel{p}\leftarrow \sum_{j = 1 }^{\infty}\mathbb{Z},$ where $p$
acts by the rule $p(a_1, a_2, a_3, \cdots) = (0, a_1, a_2, a_3,
\cdots),\ \ $ therefore
$$\underleftarrow{\rm{lim}}^{(1)}\check{H}^{1}(U_i)\cong
\underleftarrow{\rm{lim}^{(1)}}(\Sigma_{1}^{\infty}Z\stackrel{p}\leftarrow
\Sigma_{1}^{\infty}Z).$$

We have following commutative diagram with exact rows:

$$\begin{array}{ccccccccc}
0&\rightarrow& \sum_{1}^{\infty}\mathbb{Z} &\stackrel{q}\rightarrow& \sum_{1}^{\infty}\mathbb{Z} &\rightarrow & 0 &&    \\
&&p \uparrow && q\uparrow && l_1\uparrow &&    \\
0&\rightarrow& \sum_{1}^{\infty}\mathbb{Z} &\stackrel{p}\rightarrow& \sum_{1}^{\infty}\mathbb{Z} &\rightarrow & \mathbb{Z}&\rightarrow &  0  \\
&&p \uparrow && q\uparrow && l_2\uparrow &&    \\
0&\rightarrow& \sum_{1}^{\infty}\mathbb{Z} &\stackrel{p^2}\rightarrow& \sum_{1}^{\infty}\mathbb{Z} &\rightarrow & \mathbb{Z\bigoplus \mathbb{Z}}&\rightarrow &  0  \\
&&p \uparrow&& q\uparrow && l_3\uparrow &&
\end{array},$$
where $q$ is the
identity mapping and $l_n: \sum_{1}^n\mathbb{Z} \rightarrow
\sum_{1}^{n-1}\mathbb{Z} $ is the
projection defined by
$$l_n(a_1,
a_2,\cdots, a_{n-1}, a_n) = (a_1, a_2,\cdots, a_{n-1}).$$

For the  inverse limit functor and its first derived functor we
have following exact sequence \cite[Property 5 ]{Kh}):
$$0\rightarrow \underleftarrow{\rm{lim}} (\Sigma_{1}^{\infty}\mathbb{Z}\stackrel{p}\leftarrow \Sigma_{1}^{\infty}\mathbb{Z})\rightarrow
\underleftarrow{\rm{lim}}(\Sigma_{1}^{\infty}\mathbb{Z}\stackrel{q}\leftarrow
\Sigma_{1}^{\infty}\mathbb{Z})\rightarrow
\underleftarrow{\rm{lim}}(\Sigma_{1}^{n-1}\mathbb{Z}\stackrel{l_n}\leftarrow
\Sigma_{1}^{n}\mathbb{Z})\rightarrow$$
$$\rightarrow
\underleftarrow{\rm{lim}^{(1)}}(\Sigma_{1}^{\infty}\mathbb{Z}\stackrel{p}\leftarrow
\Sigma_{1}^{\infty}\mathbb{Z})\rightarrow
\underleftarrow{\rm{lim}^{(1)}}(\Sigma_{1}^{\infty}\mathbb{Z}\stackrel{q}\leftarrow
\Sigma_{1}^{\infty}\mathbb{Z})\rightarrow
\underleftarrow{\rm{lim}^{(1)}}(\Sigma_{1}^{n-1}\mathbb{Z}\stackrel{l_n}\leftarrow
\Sigma_{1}^{n}\mathbb{Z})\rightarrow 0. $$

Obviously,
$$\underleftarrow{\rm{lim}} (\Sigma_{1}^{\infty}\mathbb{Z}
\stackrel{p}\leftarrow \Sigma_{1}^{\infty}\mathbb{Z})\cong 0,$$
$$\underleftarrow{\rm{lim}}(\Sigma_{1}^{\infty}\mathbb{Z}\stackrel{q}\leftarrow
\Sigma_{1}^{\infty}\mathbb{Z}) \cong \Sigma_{1}^{\infty}\mathbb{Z},$$
$$\underleftarrow{\rm{lim}}(\Sigma_{1}^{n-1}\mathbb{Z}\stackrel{l_n}\leftarrow
\Sigma_{1}^{n}\mathbb{Z}) \cong \prod_1^{\infty}\mathbb{Z}$$ and
$$\underleftarrow{\rm{lim}^{(1)}}(\Sigma_{1}^{\infty}\mathbb{Z}\stackrel{q}\leftarrow
\Sigma_{1}^{\infty}\mathbb{Z})\cong 0 $$ since $q$ is the
identity mapping.

It follows from this diagram that
$$\underleftarrow{\rm{lim}^{(1)}}(\Sigma_{1}^{\infty}\mathbb{Z}\stackrel{p}\leftarrow
\Sigma_{1}^{\infty}\mathbb{Z}) \cong 
(\prod_{i = 1}^{\infty}\mathbb{Z})
/ (\sum_{i = 1}^{\infty}\mathbb{Z}),$$
therefore
$$\underleftarrow{\rm{lim}}^{(1)}\check{H}^{1}(U_i)\cong
(\prod_{i = 1}^{\infty}\mathbb{Z})
/ (\sum_{i = 1}^{\infty}\mathbb{Z}).$$

Since $U_i$ are homotopy equivalent to 1-dimensional space
(Hawaiian earring) it follows that $\underleftarrow{\rm{lim}} \
\check{H}^2(U_i) \cong 0$ and by  Proposition \ref{Inverse Lim}
we have
$$\check{H}^2(HA) \cong 
(\prod_{i = 1}^{\infty}\mathbb{Z})
/ (\sum_{i = 1}^{\infty}\mathbb{Z}).$$

Since, as it was mentioned in the proof of  Theorem \ref{H^2(HA)},
$\underleftarrow{\rm{lim}}\ \check{H}^{1}(U_i)\cong 0$ it follows
from the
exact sequence $$0\rightarrow
\underleftarrow{\rm{lim}}^{(1)}\check{H}^{0}(U_i) \rightarrow
\check{H}^1(HA)\rightarrow \underleftarrow{\rm{lim}} \
\check{H}^1(U_i)\rightarrow 0$$ that $\check{H}^1(HA)
\cong0.$
Since $\dim HA = 2$ it follows that $\check{H}^n(HA)
\cong 0,$ for all $n > 2.$

\end{proof}

\begin{rmk}
From the homotopical point of view, the spaces $\mathcal{G}$ and
$\mathcal{HA}$ are very close to each other -- it is possible to
show that  $\mathcal{G}$and $\mathcal{HA}$ are homotopy
equivalent to $\mathcal{MG}$ and $\mathcal{MHA}$, respectively.
However,
by definition of  $\mathcal{MHA},$ 
it follows that
$$\mathcal{MHA} = \mathcal{MG} \setminus \{a, b\},$$ for some pair of
points $a, b$ (see Figures 2 and 3).
\end{rmk}

\section{Acknowledgements}
This research was supported by the Slovenian Research Agency
grants P1-0292-0101, J1-2057-0101 and J1-4144-0101. We thank
Katsuya Eda and the referee for their
comments and suggestions.

\providecommand{\bysame}{\leavevmode\hbox to3em{\hrulefill}\thinspace}

\end{document}